\documentclass{amsart}

\usepackage{amssymb,color}
\usepackage{amsfonts}
\usepackage{amsmath}
\usepackage{euscript}
\usepackage{enumerate}
\usepackage{graphics}

\newtheorem{theorem}{Theorem}[section]
\newtheorem{lemma}[theorem]{Lemma}
\newtheorem{note}[theorem]{Note}
\newtheorem{prop}[theorem]{Proposition}
\newtheorem{cor}[theorem]{Corollary}

\newtheorem*{Theorem1'}{Theorem 1'}

\theoremstyle{definition}

\newtheorem{example}[theorem]{Example}

\theoremstyle{remark}

\numberwithin{equation}{section}

\newcommand \Q{{\mathbb Q}}

\newcommand \End{{\mathrm {End}}}

\newcommand \chr{{\mathrm {char}}}

\newcommand \Z{{\mathbb Z}}
\newcommand \GL{{\mathrm {GL}}}

\newcommand \al{{\alpha}}

\newcommand \de{{\delta}}
\newcommand \ga{{\gamma}}

\newcommand \si{{\sigma}}

\newcommand \Aut{{\mathrm{ Aut}}}

\newcommand \Ann{{\mathrm{ Ann}}}

\begin{document}

\title [\small{Groups admitting a faithful irreducible
representation}] {\small{On infinite groups admitting a faithful irreducible
representation}}

\author{Fernando Szechtman}
\address{Department of Mathematics and Statistics, University of Regina, Regina, Canada}
\email{fernando.szechtman@gmail.com}
%\thanks{The first and second authors were supported in part by an NSERC discovery grant}

\author{Anatolii Tushev}
\address{Department of Mathematics, Dnepropetrovsk National University, Dnepropetrovsk, Ukraine}
\email{tushev@member.ams.org}

%    General info


\begin{abstract} Necessary and sufficient conditions for a group to possess a faithful irreducible
representation are investigated. We consider locally finite groups and groups whose socle is essential.
\end{abstract}

\maketitle

\section{Introduction}

In 1911, Burnside \cite[Note F]{B} posed the problem of finding necessary and sufficient conditions
for a finite group to admit a faithful irreducible complex representation. Several solutions have appeared since then.
See \cite{Sz} for a detailed history of the subject.

The general case of Burnside's problem, when the group and the ground field are arbitrary, is considerably
harder. Partial solutions appear in \cite{T} and \cite{Sz}. It should be noted that \cite[Theorem 1]{T} as well
as \cite[Theorem 1.2]{Sz} require finiteness conditions. In this paper we obtain the following general criterion, where
all finiteness conditions have been removed.

We fix throughout an arbitrary field $k$. All modules will be
assumed to be from the left.

Given a group $G$, let $\mathrm{soc}(G)$ stand for the subgroup of
$G$ generated by all minimal normal subgroups of $G$, and
$\mathrm{A}(G)$ for the subgroup of $G$ generated by all minimal
normal subgroups of $G$ that are torsion abelian. A normal subgroup $N$ of
$G$ is said to be essential if every nontrivial normal subgroup of
$G$ intersects $N$ nontrivially. Recall that $G$ is locally cyclic (resp. finite) if every finitely
generated subgroup of $G$ is cyclic (resp. finite), and that $\mathrm{core}_G(C)$ stands
for the intersection of all normal subgroups of $G$ contained in a given subgroup $C$ of $G$. We write $\Pi(G)$ for
the set of all primes $p$ such that $a^p=1$ for some $a\in G$, $a\neq 1$.

\begin{theorem}\label{tres} Let $G$ be an arbitrary group such that $\mathrm{soc}(G)$ is essential and $\chr(k)\notin\Pi(\mathrm{soc}(G))$.
Then $G$ has a faithful irreducible representation over $k$
if and only if there is a
subgroup $C$ of $\mathrm{A}(G)$ such that $\mathrm{A}(G)/C$ is
locally cyclic and $\mathrm{core}_G(C)=1$.
\end{theorem}

{\color{blue} For locally finite groups we have the following stronger version, where $\Pi'(G)$ stands set of all primes $p$ such that $G$ has
a nontrivial normal $p$-subgroup.

\begin{theorem}\label{nuet} Let $G$ be a locally finite group where $\mathrm{soc}(G)$ is essential. Then $G$ has a faithful irreducible representation over $k$ if and only if $\chr(k)\notin\Pi'(G)$ and there is a subgroup $C$ of $\mathrm{A}(G)$ such that $\mathrm{A}(G)/C$ is locally cyclic and~$\mathrm{core}_G(C)=1$.
\end{theorem}
}

For sufficiency in both cases we rely on \cite[Theorem 1.1]{Sz}. We are more concerned here with the necessity of the stated
conditions. In this regard, our main tool to prove Theorem \ref{tres} is the following result, stated under no finiteness conditions.

\begin{theorem}\label{la} Let $G$ be a group with a torsion abelian
subgroup~$A$. Suppose that $\mathrm{char}(k)\notin\Pi(A)$ and $G$
admits a faithful irreducible module $V$ over~$k$.
Then there is a subgroup $C$ of $A$ such that $A/C$ is locally cyclic and $\mathrm{core}_G(C)=1$.
\end{theorem}

The following additional tool is required to prove Theorem \ref{nuet}.

\begin{theorem}\label{te5} Suppose that $k$ has prime characteristic $p$ and let $G$ be a locally finite group with a normal $p$-subgroup $P$.
Then $P$ acts trivially on every irreducible $kG$-module~$V$.
\end{theorem}

Theorems \ref{la} and \ref{te5} also allow us to obtain results on irreducible representations of soluble torsion groups.

\begin{theorem}\label{soltor} Let $G$ be a torsion soluble group.

(a) If $G$ has a faithful irreducible representation over $k$, then for any abelian normal subgroup
$A$ of $G$ we have   $\chr(k)\notin\Pi(A)$   and there is a subgroup $C$ of $A$ such that $A/C$ is
locally cyclic and $\mathrm{core}_G(C)=1$.

(b) If $G$ has a maximal abelian normal subgroup $A$  such that $\chr(k)\notin\Pi(A)$
and there is a subgroup $C$ of $A$ such that $A/C$ is locally cyclic and $\mathrm{core}_G(C)=1$, then
$G$ has a faithful irreducible representation over $k$.
\end{theorem}

Evidently, the necessary and sufficient conditions obtained in the above theorem may be amalgamated in the
following criterion.

\begin{cor}\label{soltor1} Let $G$ be a torsion soluble group.

(a) If $G$ is abelian, then $G$ has a faithful irreducible representation over $k$ if and only if $\chr(k)\notin\Pi(G)$
and G is locally cyclic.

(b) If $G$ is nonabelian, then $G$ has a faithful irreducible representation over $k$ if and only if $G$ has a maximal abelian normal subgroup $A$ such that $\chr(k)\notin\Pi(A)$ and there is a subgroup $C$ of $A$ such that $A/C$ is locally cyclic and $\mathrm{core}_G(C)=1$.
\end{cor}

The discussion is illustrated by several examples. In particular, if $k$ has prime characteristic $p$ and $G=F_p(t)^+\rtimes F_p(t)^*$, then
$G$ is shown to have a faithful irreducible module over $k$. As indicated in \S\ref{ejer}, this feature of $F_p(t)^+\rtimes F_p(t)^*$ serves as counterexample to numerous results when their respective finiteness conditions are removed. This applies, in particular, to Theorem
\ref{nuet}, where local finiteness cannot be eliminated, to Theorem
\ref{te5}, whose hypothesis cannot be relaxed to ask that $P$, but not necessarily
$G$, be locally finite, and to the following necessary criterion for a group to have a faithful irreducible representation.

\begin{theorem}\label{te1} Let $G$ be a group having a faithful irreducible module $V$ over~$k$. Let $A$ be a normal, torsion, abelian subgroup of $G$. Suppose
that
\begin{equation}
\label{alguno} [G:C_G(A)]<\infty.
\end{equation}
Then $\chr(k)\notin \Pi(A)$.
\end{theorem}

The finiteness condition (\ref{alguno}) is imposed to ensure that
$V$ has an irreducible $kC_G(A)$-submodule,  which implies the existence an irreducible $kA$-submodule of~$V$. In general, no such
submodules exist.

\section{Proofs}

To prove Theorem \ref{te5} we will rely on the following well-known special case.

\begin{theorem}\label{te4} Suppose $k$ has prime characteristic $p$ and let $G$ be a group with a finite, normal, $p$-subgroup $P$.
Then $P$ acts trivially on every irreducible $kG$-module~$V$.
\end{theorem}

\noindent{ \textbf{\textit{ Proof of Theorem \ref{te5}}}}. Let $a\in P$. We need to show that $1-a\in J(kG)$. This means that $1-\beta(1-a)$ is left invertible in $kG$ for every $\beta\in kG$. Given such $\beta$,
let $H$ be the subgroup of $G$ generated by $a$ and the support of $\beta$. Since $G$ is locally finite and $H$ is finitely generated, $H$ is a finite subgroup of $G$.
But $a\in P\cap H$, which is a normal $p$-subgroup
of the finite group~$H$. By Theorem \ref{te4}, $P\cap H$ acts
trivially on every irreducible $kH$-module, so $1-a\in J(kH)$.
Thus $1-\beta(1-a)$ is left invertible in $kH$ and hence in $kG$.\qed

\medskip

The proof of Theorem \ref{te1} will require the following obvious result.

\begin{lemma}\label{le1} Let $N\unlhd G$ be groups and let $V$ be an irreducible $kG$-module. Suppose $N$ has a nonzero fixed point in $V$.
Then $N$ acts trivially on $V$.
\end{lemma}

Given an abelian group $A$ and a prime $p$, we let $A_p$ stand for
the subgroup of $A$ consisting of all $a\in A$ such that
$a^{p^m}=1$ for some $m\geq 0$.

\medskip

\noindent{\textbf{ \textit{ Proof of Theorem \ref{te1}}}}. Since $[G:C_G(A)]<\infty$, \cite[Theorem 7.2.16]{Pa}
ensures the existence of an irreducible $k C_G(A)$-submodule, say
$U$,  of $V$. Since $U$ is irreducible, $D=\End_{k C_G(A)}(U)$ is a division ring.

Suppose $k$ has prime characteristic $p$. Given $a\in A_p$, the element $1-a\in kA_p$ is nilpotent, so
its image in $D$ must be 0. It follows that $A_p$ acts trivially on $U$. As $A_p$ is a characteristic
subgroup of $A$, we infer from Lemma \ref{le1} that $A_p$ acts
trivially on $V$. Since $V$ is a faithful $G$-module, we deduce that
$A_p$ is trivial.\qed

\medskip

We will make use of the following two results in order to prove Theorem \ref{la}.

%\begin{lemma}\label{co1} Let $G$ be a group having a faithful irreducible module $V$ over~$k$. Let $C$ be any subgroup of $G$
%having a nonzero fixed point in $V$. Then $\core_G(C)=1$.
%\end{lemma}

%\begin{proof} This follows at once from Lemma \ref{le1}.
%\end{proof}

\begin{lemma}\label{tu} Let $A$ be a torsion abelian group. Let $J$ be a maximal ideal of $kA$ and set $J^\dag = (J+1) \bigcap A$.
Then the quotient group $A/J^\dag$ is locally cyclic.
\end{lemma}

\begin{proof} As $J$ is maximal, $F= kA/J$ is a field and we have a natural homomorphism $A\to F^*$
with kernel $J^\dag$. As the torsion subgroup of $F^*$  is locally cyclic, the assertion follows.
\end{proof}

\begin{prop}\label{tu2} Let $A$ be a torsion abelian group. Suppose that $\mathrm{char}(k)\notin\Pi(A)$. Then for any
$kA$-module $V\neq 0$ there is a maximal ideal $J$ of $kA$ such that $JV\neq V$.
\end{prop}

\begin{proof} Given $0\neq a \in V$, Zorn's Lemma ensures the existence of a submodule $U$ of $V$ that is maximal with respect to $a\notin U$. It is sufficient to show that $J=\Ann_{kA}(V/U)$ is a maximal ideal of $kA$.

Suppose $J$ is not maximal and hence $kA/J$ is not a field. This easily implies that there is a finite subgroup $X$ of $A$ such that  $kX/J_X$ is not a field, where $J_X= kX \bigcap J = \Ann_{kX} (V/U)$, and hence  $J_X=\Ann_{kX} (V/U)$ is not a maximal ideal of $kX$.

Since $\mathrm{char}(k)\notin\Pi(A)$, we have $1=e_1+\cdots+e_m$, where $e_1,\dots,e_m$ are orthogonal idempotents in $kX$ and each $S_i=e_ikX$
is a simple ideal of $kX$. We then have $V/U=e_1\cdot (V/U)\oplus\cdots\oplus e_n\cdot (V/U)$, where the annihilator of each nonzero component $e_i\cdot (V/U)$ is a maximal ideal of $kX$, namely the sum of all $S_j$ with $j\neq i$. As $\Ann_{kX} (V/U)$ is not a maximal ideal of $kX$,
there must be at least two nonzero components. On the other hand, since $a\notin U$, we see that $a+U$ can belong
to at most one component. Thus there is at least one $i$ such that $e_i(V/U)\neq 0$ and $a+U\notin e_i(V/U)$. As $kA$ is commutative, we readily see that $e_iV+U$ is a $kA$-submodule of $V$. As such, it properly contains $U$ and $a\notin e_iV+U$, against the maximality of $U$.
Thus, a contradiction is obtained and hence $J=\Ann_{kA}(V/U)$ is a maximal ideal of $kA$.
\end{proof}

Example \ref{jt1} shows that the condition $\mathrm{char}(k)\notin\Pi(A)$ cannot be removed from Proposition \ref{tu2}.

\medskip

\noindent{ \textbf{ \textit{ Proof of Theorem \ref{la}}}}. By Proposition \ref{tu2}, there is a maximal ideal $J$ of $kA$ such that $JV\neq V$. Set $C=J^\dag$, as in Lemma \ref{tu}, which ensures that the quotient group $A/C$ is locally cyclic.

 Let $N =\mathrm{core}_G(C)$. It follows from the definition of $J^\dag$ that $(1-J^\dag) \subseteq J$ and hence, as $N\leq C=J^\dag$, we have $(1-N) \subseteq J$. Therefore, as $JV\neq V$, we have  $(1-N)V\neq V$. Since $N$ is a normal subgroup of $G$, we can show that $(1-N)V$ is a $kG$-submodule of $kG$. Then, as the module $V$ is irreducible and $(1-N)V\neq V$, we have $(1-N)V=0$, which means that $N$ acts trivially on~$V$.  As $V$ is faithful, we conclude that $N=1$.\qed

\medskip

We have not been able to determine whether
the condition $\mathrm{char}(k)\notin\Pi(A)$ can be removed from Theorem \ref{la}, specially when $A$ is a normal subgroup of of $G$ and $[G:A]$ is not necessarily finite.

\medskip

\noindent{ \textbf{ \textit{ Proof of Theorem \ref{soltor}}}}. (a) It is well-known that a torsion soluble group must be locally finite.
It therefore follows from Theorem \ref{te5} that $\chr(k)\notin \Pi(A)$. We may now apply Theorem \ref{la} to derive the existence of the required subgroup $C$.

(b) As $A$ is a maximal abelian normal subgroup of a soluble group $G$, it is easy to see that $A$ is essential in $G$. By \cite[Lemma 2]{T}, $A$ has an irreducible representation $\varphi$ over $k$ such that
$Ker\, \varphi = C$ . Since $C = Ker \varphi $  contains no nontrivial normal subgroups of $G$, we see that  $A\cap X $ is not contained in
$Ker\, \varphi $ for any nontrivial normal subgroup $X$ of $G$, so the assertion follows from \cite[Lemma 3]{T}.

\medskip

We are finally in a position to prove Theorems \ref{tres} and \ref{nuet}.

\medskip

\noindent{ \textbf{ \textit{ Proof of Theorem \ref{tres}}}}. Suppose first that $G$ has a faithful irreducible $G$-module over $k$. Then, by
Theorem \ref{la}, there is a subgroup $C$ of $\mathrm{A}(G)$ such that
$\mathrm{A}(G)/C$ is locally cyclic and $\mathrm{core}_G(C)=1$. Suppose, conversely, that this condition holds.
Then $G$ has a faithful irreducible $G$-module over $k$ by
\cite[Theorem 1.1]{Sz}.\qed

\medskip

{\color{blue}
\noindent{ \textbf{\textit{ Proof of Theorem \ref{nuet}}}}. Necessity follows from Theorems \ref{tres} and \ref{te5}, while
sufficiency is consequence of \cite[Theorem 1.1]{Sz}.
\qed
}

\section{Examples}\label{ejer}

The following three results are well-known.

\begin{theorem}\label{a1} Let $N\leq G$ be groups. Then

\noindent (a) $J(kG)\cap kN\subseteq J(kN)$.

\noindent (b) If $N\unlhd G$ and $N$ or $[G:N]$ are finite, then
$$
J(kN)=J(kG)\cap kN.
$$

\noindent (c) (Clifford) If $N\unlhd G$ and $N$ or $[G:N]$ are finite, then
every irreducible $kG$-module is a completely reducible $kN$-module.
\end{theorem}

\begin{theorem}\label{b1} (Lie-Kolchin) Let $V$ be a nonzero finite dimensional vector space over $k$ and let $A$ be a subgroup of $\GL(V)$
consisting of unipotent operators. Then $A$ has a nonzero fixed point in $V$.
\end{theorem}

%\begin{lemma}\label{h1} Let $H\leq G$ be groups. Then $J(kG)\cap kH\subseteq J(kH)$.
%\end{lemma}

%\begin{theorem}\label{c1} Let $A\unlhd G$ be groups, with $A$ or $[G:A]$ finite. Then
%$$
%J(kA)=J(kG)\cap kA.
%$$
%\end{theorem}

\begin{theorem}\label{nak} (Nakayama) Let $R$ be a ring and let $V$ be a finitely generated $R$-module
satisfying $J(R)V=V$. Then $V=0$.
\end{theorem}

\begin{example}\label{jt1} Suppose $k$ has prime characteristic $p$ and let $G=C_p\wr C_\infty$ be the wreath product of a cyclic group $C_p$ of order $p$ by an
infinite cyclic group $C_\infty$. Then $G$ has a faithful irreducible module over $k$.
\end{example}

\begin{proof} It is shown in \cite[Lemma 9.2.8]{Pa} that $kG$ has a faithful irreducible module, which is automatically a faithful $G$-module over $k$.

\end{proof}

\begin{note}{\rm The Jacobson radical of the group algebra $kG$ of Example \ref{jt1} was studied first in \cite{W}.}
\end{note}

Example \ref{jt1} illustrates the failure of several results when
the stated finiteness conditions are removed. Theorems
\ref{te5}, \ref{te1} and \ref{te4} above provide examples of this
phenomenon. Further examples are given by Theorem \ref{a1}, parts (b) and (c),
as well as Theorems \ref{b1} and \ref{nak}. Indeed, let $A$ be a vector space over $F_p$ with basis $(e_i)_{i\in\Z}$ and let $g\in\Aut(A)$ be given by $g(e_i)=e_{i+1}$ for $i\in\Z$, so that $G=A\rtimes \langle g\rangle$. Let $V$ be the faithful irreducible
$G$-module over $k$ ensured by
Example \ref{jt1}. If $V$ had an irreducible $kA$-submodule $U$
then $A$ would act trivially on $U$ by Theorem \ref{te5}, so $A$
would act trivially on $V$ by Lemma \ref{le1}, contradicting the
fact that $V$ is faithful. Moreover, it is clear that $A$ acts on
$V$ via unipotent operators and, as just indicated, $A$ has no
nonzero fixed points in $V$. Furthermore, we have $J(kG)\cap
kA=0$, by  the primitivity of $kG$, whereas Theorem \ref{te5} implies that
$J(kA)$ is the augmentation ideal of $kA$. In this regard, since $J(kA)V$
is a $kG$-submodule of $V$ and $A$ does not act trivially on $V$, we have $J(kA)V=V$.
As $J(kA)$ is the {\em only} maximal ideal of $kA$, this shows, in addition, that the condition
$\mathrm{char}(k)\notin\Pi(A)$ cannot be removed from Proposition \ref{tu2}.

\begin{prop}\label{h2} Let $G$ be a group with subgroups $H$ and $N$ satisfying:

\noindent (H1) $H$ has a faithful irreducible representation over $k$.

\noindent (H2) $N$ is normal and nontrivial, and is included in every
nontrivial normal subgroup of $G$.

\noindent (H3) $N\cap H\neq 1$.

Then $G$ has a faithful irreducible representation over $k$.
\end{prop}

\begin{proof} By (H3), there is a nontrivial $x\in N\cap H$. By (H1), $x-1\notin J(kH)$, so $x-1\notin J(kG)$ by Theorem \ref{a1}(a).
Thus, there is an irreducible $kG$-module $V$ not acted upon trivially by $x$. It follows from (H2) that $V$ is a faithful $G$-module.
\end{proof}

\begin{example}\label{e1} Suppose $k$ has prime characteristic $p$. Let $D$ be a division algebra of characteristic $p$ that is not
algebraic over its prime field. Then $G=D^+\rtimes D^*$ has a faithful irreducible module over $k$.
\end{example}

\begin{proof} Take $N=D^+$ and $H=\langle t^i\,|\, i\in \Z\rangle \rtimes \langle t\rangle\leq D^+\rtimes D^*$
in Proposition \ref{h2}, where $t\in D$ is transcendental over $F_p$, and use Example \ref{jt1}.
\end{proof}

\begin{example}\label{e2} Suppose $k$ has prime characteristic $p$. Let $M=M(\Q,\leq, F_p)$ be the  McLain group, as defined in \cite{M},
and set $G=M\rtimes \Aut(M)$. Then $G$ has a faithful irreducible module over $k$.
\end{example}

\begin{proof} Let $A$ be the subgroup of $M$ generated by $1+e_{2i,2i+1}$, $i\in\Z$, let $g\in\Aut(M)$ be induced by the
translation $i\mapsto i+2$ of $\Q$. Take $H=A\rtimes \langle g\rangle$ and $N=M$ in Proposition \ref{h2} and use Example \ref{jt1}.
\end{proof}

Since $D^+$ and $M$ are locally finite $p$-groups, $D^+\rtimes
D^*$ and $M\rtimes \Aut(M)$ also exhibit the failure of Theorems
\ref{te4}, \ref{te5}, \ref{te1}, \ref{a1}, \ref{b1} and \ref{nak}
when the required finiteness conditions are removed (except that
$M\rtimes \Aut(M)$ is unrelated to Theorem \ref{te1} as $M$ is not
abelian). Moreover, Example \ref{e1} shows that \cite[Theorem
1]{T} and \cite[Theorem 1.2]{Sz} cease to be valid if their
respective finiteness conditions are removed. Indeed, in Example
\ref{e1} the given group $G$ has a minimal normal subgroup $A$
that is an abelian $p$-group but nevertheless $G$ has a faithful
irreducible module over a field of characteristic $p$ (in this
regard, note that the socle of $C_p\wr C_\infty$ is trivial).

%\section{Still unclear}

%\begin{ques} Can we remove the condition $\chr(k)\notin\Pi(A)$
%from Theorem \ref{la}? What if $A$ is normal? How about if $A$ is
%the direct sum of isomorphic irreducible $\Z G$-modules, with $G$
%acting by conjugation?
%\end{ques}

\end{document}